\documentclass[12pt,nointlimits,oneside]{elsarticle}
\usepackage{amsmath,amssymb,amsfonts}
\usepackage{amsthm}

\numberwithin{equation}{section} \theoremstyle{plain}
\newtheorem{theorem}{Theorem}[section]

\newtheorem{lemma}[theorem]{Lemma}
\newtheorem{proposition}[theorem]{Proposition}
\theoremstyle{definition}
\newtheorem{definition}[theorem]{Definition}
\theoremstyle{remark}
\newtheorem{remark}[theorem]{Remark}

\newcommand{\set}[1]{\left\{#1\right\}}
\newcommand{\norm}[1]{\left\lVert#1\right\rVert}
\newcommand{\spr}[1]{#1}
\newcommand{\ind}[1]{1\hspace{-.28em}\mathrm{I}_{#1}}
\newcommand{\abs}[1]{\left\vert#1\right\vert}
\newcommand{\ex}[1]{\mathsf{E}\left[\,#1\,\right]}
\newcommand{\exxi}[1]{\mathsf{E}_\xi\left[\,#1\,\right]}

\DeclareMathOperator{\sign}{sign}
\newcommand{\R}{{\mathbb R}}
\newcommand{\F}{\mathcal F}

\newcommand{\la}{\lambda}
\DeclareMathOperator{\re}{Re}

\newcommand{\wH}{\hat{H}}
\newcommand{\cH}{\check{H}}
\newcommand{\sas}{S$\alpha$S}
\DeclareMathOperator{\spann}{span}

\begin{document}

\title{Real harmonizable multifractional stable process and its local properties\tnoteref{t1}}

\author{Marco Dozzi}
\ead{marco.dozzi@iecn.u-nancy.fr}
\address{Institut Elie Cartan, Universit\'e Henri Poincar\'e Nancy 1,
B.P. 239, F-54506 Vandoeuvre-l\`es-Nancy Cedex}%

\author{Georgiy Shevchenko\corref{cor1}}
\ead{zhora@univ.kiev.ua}
\cortext[cor1]{Corresponding author}

\address{Kiev National Taras Shevchenko University, Department of Mechanics and Mathematics, Volodymyrska str. 64, 01601, Kiev, Ukraine}%

\tnotetext[t1]{This work has been partially supported by the Commission of the European Committees Grant PIRSES-GA-2008-230804 within the program ``Marie Curie Actions''.}
\begin{abstract}
A real harmonizable multifractional stable process is defined, its H\"older continuity and localizability are proved.
The existence of local time is shown and its regularity is established.
\end{abstract}
\maketitle

\section*{Introduction}
Fractional processes are one of the main tools for  modeling the phenomena of long-range dependence in natural sciences, financial
mathematics, telecommunication networks etc. Due to the role played by Gaussian distribution, the most popular and
the most intensively investigated fractional process is the fractional Brownian motion $B^H$, a centered Gaussian process with
the covariance function $\ex{B^H(t) B^H(s)} = \frac12 (t^{2H}+s^{2H}-\abs{t-s}^{2H})$. The parameter $H\in (0,1)$ is called the Hurst
parameter and measures the smoothness of trajectories of the process (it is approximately the H\"older exponent of the process) and
the ``depth of memory'' of the process (for $H>1/2$ the process exhibits the property of long-range dependence).

From the point of view of possible applications, there are two main drawbacks of fractional Brownian motion. The first one comes from the Gaussian
distribution, which has extremely light tails, though many data coming from applications are heavy-tailed. The second one is the homogeneity
of increments that does not allow to model processes having different regularity and different time dependence properties at different
time instances. A related problem is a self-similarity property, which briefly means that the properties of the process are the same
under each scale. However, the absence of such property is apparent in many cases and mostly evident in stock price processes: long-term
data is much smoother than wild intraday quotes.

The light tails problem is worked around usually by considering fractional stable processes. In contrast to the Gaussian case, where
the covariance structure determines the whole distribution of a process, so there is essentially one fractional process, in stable case
there are many of them: linear fractional stable process, harmonizable fractional stable process, Liouville stable process etc. (See
book \cite{SamorTaqqu} for an extensive review of different fractional processes.)

In turn, the homogeneity problem is solved by considering multifractional processes. Recently, several multifractional extensions
of fractional Brownian motion were defined, based on different representations of the fractional Brownian motion: moving average
(linear) multifractional Brownian motion \cite{PeltierLevyVehel}, Volterra multifractional Brownian motion \cite{ralchenko},
harmonizable multifractional Brownian motion \cite{BenassiJaffardRoux97}.

In this paper, we consider a process called \emph{real harmonizable multifractional stable process} which has both properties
of heavy tails and multifractionality, which can be regarded
both as a multifractional generalization of a harmonizable fractional stable process and as a stable generalization of
harmonizable multifractional Brownian motion, and can be used to improve models involving either kind of processes.

Our main interest in this paper is in path properties of this process: continuity, existence and joint continuity of local times.
For fractional harmonizable stable process continuity was proved in \cite{kono-maejima} and local times properties were considered in
\cite{Xiao}. For related results with multifractional harmonizable L\'evy processes   we refer to \cite{lacaux04}.

The paper is organized as follows. In Section 1 we give necessary pre-requisites on stable distributions and local times. Section 2 focuses
on path properties of the process considered: almost sure continuity and localizability. Section 3 is devoted to existence and properties
of local times.

\section{Pre-requisites}\label{sec:pre-req}
\subsection{Stable random variables and processes}\label{subsec:stable}
In this paper we focus only on symmetric $\alpha$-stable (\sas) random variables with $\alpha\in(1,2)$. We recall that a random variable $\xi$
is called \sas\ with a scale parameter $\sigma^\alpha$ if it has a characteristic function
$$
\ex{e^{i\lambda \xi}}= e^{-\abs{\sigma\lambda}^\alpha}.
$$

An important tool to construct stable random variables is \emph{independently scattered rotationally invariant complex \sas\ random measure} with the Lebesgue control measure, which is a complex-valued $\sigma$-additive random measure $M=M_\alpha$ on $\R$ defined by the following properties.
\begin{enumerate}
\item (\emph{Rotationally invariant complex \sas}) for any Borel set $A\subset \R$ and any $\theta\in \R$ the distribution of $e^{i\theta}M(A)$ is the
same as of $M(A)$, and $\mathrm{Re}\, M(A)$ is \sas\ with the scale parameter $\lambda(A)$.
\item (\emph{Independently scattered}) for any disjoint Borel sets $A_1,\dots,A_n\subset[0,\infty)$ the values $M(A_1),\dots, M(A_n)$ are
 independent.
\item For any Borel set $A\subset \R$ \ $M(-A) = \overline{M(A)}$.
\end{enumerate}

For a function $f:\R\to \mathbb C$ such that
\begin{equation}
\label{kindofsymmetry}
f(-x)=\overline{f(x)}\text{ for all }x\in \R
\end{equation}
 and
$$\norm{f}^\alpha_{L^\alpha(\R)}=\int_\R \abs{f(x)}^{\alpha}dx<\infty$$  it is possible to define a stochastic
integral
$$
\int_R f(x) M(dx),
$$
which appears to be a real \sas\ random variable with the scale parameter $\norm{f}^\alpha_{L^\alpha(\R)}$.

In other words, stochastic integral gives an isometry between the space of \sas\ real random variables spanned by the measure $M$
with the norm $$\norm{\xi}_\alpha^\alpha = c(\xi)=-\log \ex{e^{i\xi}}$$
and the subspace of $L^\alpha(\R)$ consisting of functions satisfying \eqref{kindofsymmetry}, i.e. having
adjoint values at symmetric points.

We end this subsection with the so-called LePage representation of processes given as transformations of \sas\ random measure.
For details see \cite{marcus-pisier, kono-maejima}.

Assume we have a measurable function $f\colon\R_+\times \R\to \mathbb C$ such that for each $t\ge 0$ the function $f(t,\cdot)$ satisfies
\eqref{kindofsymmetry} and belongs to $L^\alpha(\R)$. Define a process $\set{X(t),t\ge 0}$ by
\begin{equation}
\label{sasprocess}
X(t) = \int_\R f(t,x) M(dx).
\end{equation}

The next proposition is a slight modification of \cite{KM}, the proof is exactly the same as there with a slight adjustment for
the property $M(-dx)=\overline{M(dx)}$ in our case, so we skip it.
\begin{theorem}
Let $\varphi$ be an arbitrary probability density on $\R$ equivalent to the Lebesgue measure. Also let $\set{\Gamma_k,k\ge 1}$,
$\set{\xi_k,k\ge 1}$, $\set{g_k,k\ge 1}$ be three independent sets of random variables, such that
\begin{itemize}
\item $\set{\Gamma_k,k\ge 1}$ is a sequence of arrivals of Poisson process with unit intensity;
\item $\set{\xi_k,k\ge 1}$ is a sequence of independent random variables with density $\varphi$;
\item $\set{g_k,k\ge 1}$ are independent rotationally invariant complex Gaussian with $\ex{\abs{\mathrm{Re}\, g_k}^\alpha}=1$.
\end{itemize}

Then the process $\set{X(t),t\ge 0}$ defined by \eqref{sasprocess} has the same finite-dimensional distributions as the process
\begin{equation}\label{lepage}
X'(t) = C_\alpha \mathrm{Re}\,\sum_{k=1}^\infty \Gamma_k^{-1/\alpha} \varphi(\xi_k)^{-1/\alpha} f(t,\xi_k) g_k,
\end{equation}
where $C_\alpha = \left(\int_0^\infty x^{-\alpha} \sin x\, dx\right)^{1/\alpha}$, and this series converges almost surely
for each $t$.
\end{theorem}

\subsection{Local times}

Let $X=(X(t),t\ge0)$ be a real-valued separable random process with Borel sample
functions. For any Borel set $B\subset\mathbb{R}_{+}$ the occupation measure of $X$ on $B$ is defined by
\begin{equation*}
\mu _{B}(A)=\lambda (\{s\in B,\text{ }X(s)\in A\})\text{ \ \ \ for all
Borel sets }A\text{ in }%
\mathbb{R},
\end{equation*}
where $\lambda $ is the Lebesgue measure on $\mathbb{R}_{+}.$ If $\mu _{B}$ is absolutely continuous with respect to the Lebesgue
measure on $\mathbb{R}$, we say that $X$ has a local time on $B$ and define its local time, $%
L(B,\cdot ),$ to be the Radon-Nikodym derivative of $\mu _{B}.$ We write $%
L(t,x)$ instead of $L([0,t],x)$ and interpret it as the time spent by $X$ in
$x$ during the time period $[0,t].$

By standard monotone class arguments we deduce that the local times have a
measurable modification that satisfies the following \emph{occupation
density formula : }for any Borel set $B\subset\mathbb{R}_{+}$ and any measurable function $f:\mathbb{R}
\rightarrow\mathbb{R}$
\begin{equation*}
\int_{B}f(X(t))dt=\int_{\mathbb{R}}f(x)L(B,x)dx.
\end{equation*}

By applying this formula to $f(x)=e^{iux},$ and writing $\hat{L}(B,u)$ for the Fourier transform of $L(B,x)$, we get $\hat{L}(B,u)=\int_{B}e^{iuX(s)}ds$ and by the Fourier inversion formula $%
L(B,x)=\frac{1}{2\pi }\int \int_{B}e^{iu(X(s)-x)}dsdu,$ if this integral
exists.

As a consequence, the following expressions for the moments of local time
hold : for any $x,y\in\mathbb{R}$, $t,h\in\mathbb{R}_{+}$ and $m\geqq 2$%
\begin{gather*}
\ex{\big(L(t+h,x)-L(t,x)\big)^{m}}
\\
=\frac{1}{(2\pi )^{m}}\int_{[t,t+h]^{m}}\int_{\mathbb{R}^{m}}\exp (-ix\sum_{j=1}^{m}u_{j})\ex{\exp
\set{i\sum_{j=1}^{m}u_{j}X(s_{j})}}\prod\limits_{j=1}^{m}du_{j}\prod%
\limits_{j=1}^{m}ds_{j},
\end{gather*}

and for every even $m\geqq 2$%
\begin{eqnarray*}
&&\ex{\big(L(t+h,y)-L(t,y)-L(t+h,x)+L(t,x)\big)^{m}} \\
&=&\frac{1}{(2\pi )^{m}}\int_{[t,t+h]^{m}}\int_{\mathbb{R}^{m}}\prod\limits_{j=1}^{m}[e^{-iyu_{j}}-e^{-ixu_{j}}]\ex{\exp
\set{i\sum_{j=1}^{m}u_{j}X(s_{j})}}\prod\limits_{j=1}^{m}du_{j}\prod%
\limits_{j=1}^{m}ds_{j}.
\end{eqnarray*}

Suitable upper bounds for these moments imply the existence of a (jointly)
continuous version of local time and, as a consequence, a certain degree
of irregularity of the sample paths of the process itself. In order to prove the
joint continuity of the local time of Gaussian processes, S.M.~Berman (see e.g.
\cite{Ber70}, \cite{Ber73}) has introduced the notion of \emph{local nondeterminism }%
(LND). This notion has been extended to stable processes by J.P.~Nolan \cite{nolan},
where the equivalent notion of \emph{locally approximately
independent increments} was introduced. These notions will be recalled in section 3.
Since then the local time of stable processes has been studied by several authors; we refer to the
recent survey by Y. Xiao \cite{Xiao2} for more recent results. The local time of multifractional moving average stable processes has been studied in \cite{dozzietal}.

\section{Definition and pathwise properties of real harmonizable stable process}

Let $M$ be an independently scattered rotationally invariant complex \sas\ measure on $\R$ defined in Subsection~\ref{subsec:stable}.

Throughout the paper we will denote by $C$ any constant, which does not depend on any variables, unless otherwise is stated. Of course, $C$ may change from line to line.

Recall that a real harmonizable fractional stable process with Hurst parameter $H$ is defined as
\begin{equation}\label{zH}
Z^H(t) = \int_{\R} \frac{e^{i\spr{tx}}-1}{\abs{x}^{1/\alpha + H}} M(dx).
\end{equation}
A multifractional generalization of this definition consists, naturally, in letting the Hurst parameter depend on $t$.
\begin{definition}
A \emph{real harmonizable multifractional stable process} (rhmsp) with Hurst function $H(t)$ and a stability parameter $\alpha$
is defined as
\begin{equation}\label{X}
X(t) =  \int_{\R} \frac{e^{i\spr{tx}}-1}{\abs{x}^{1/\alpha + H(t)}} M(dx).
\end{equation}
\end{definition}
Clearly, $X(t) = Z^{H(t)}(t)$.
We assume that $0<\wH=\inf_t{H(t)}\le H(t) \le \sup_t {H(t)} = \cH<1$.

\subsection{Norm estimates for the increments}
\begin{lemma}\label{lemma1}
For all $H_1,H_2 \in (\wH,\cH)$ it holds
\begin{equation*}
\norm{Z^{H_1}(t) - Z^{H_2}(t)}_\alpha \le C\abs{H_1 - H_2},\quad t\in [0,T].
\end{equation*}
\end{lemma}
\begin{proof}
Write
\begin{gather*}
\norm{Z^{H_1}(t) - Z^{H_2}(t)}_\alpha^\alpha = \int_{\R} \abs{e^{i\spr{tx}} - 1}^\alpha |x|^{-1} \big||x|^{-H_1} - |x|^{-H_2}\big|^\alpha dx \\
\le C \int_{\R}  (1\wedge |x|)^\alpha|x|^{-1}
|\log|x||^\alpha (|x|^{-\alpha H_1}\vee |x|^{-\alpha H_2})\abs{H_1-H_2}^\alpha dx\\
= C\abs{H_1-H_2}^\alpha \left(\int_{\abs{x}<1} |x|^{\alpha(1- \wH) -1}|\log|x||^\alpha dx + \int_{\abs{x}>1} |x|^{-1 - \alpha\cH} |\log|x||^\alpha dx\right)\\
\le C\abs{H_1 - H_2}^\alpha,
\end{gather*}
whence we have the assertion.
\end{proof}

We assume  that $H$ is H\"older continuous
with order greater than $\cH$, i.e., there exists $\gamma > \cH$ s.t. for all $t,s\ge 0$
\begin{equation*}
\abs{H(t) - H(s)} \le C |t-s|^\gamma.
\end{equation*}

\begin{lemma}\label{lemma2}
There exist positive constants $C_1,C_2>0$ such that for any $H\in [\wH,\cH]$ one has
\begin{equation*}
C_1 |t-s|^{H}\le \norm{Z^H(t) - Z^H(s)}_\alpha \le C_2 |t-s|^{H}
\end{equation*}
locally uniformly in $s,t$.
\end{lemma}
\begin{proof}
Write
\begin{gather*}
\norm{Z^{H}(t) - Z^{H}(s)}_\alpha^\alpha = \int_{\R} \abs{e^{i\spr{tx}} - e^{i sx}}^\alpha |x|^{-1-\alpha H}  dx
\le C \int_{\R}  (1\wedge \abs{t-s}|x|)^\alpha|x|^{-1-\alpha H} dx\\
= C \left(\abs{t-s}^\alpha\int_{\abs{x}<1/|t-s|} |x|^{\alpha (1-H)-1} dx + \int_{\abs{x}>1/|t-s|} |x|^{-1 - \alpha H}  dx\right)\\
\le C\big(\abs{t-s}^{\alpha -\alpha(1-H)} + \abs{t-s}^{\alpha H}\big)=C\abs{t-s}^{\alpha H}.
\end{gather*}
To prove the lower bound, observe that there exist positive constants $c_1,c_2$ such that $\abs{e^{iy} - 1}>c_1\abs{y}$ for $\abs{y}<c_2$ and write
\begin{gather*}
\norm{Z^{H}(t) - Z^{H}(s)}_\alpha^\alpha \ge \int_{\abs{x}<c_2/\abs{t-s}} \abs{e^{i(t-s)x} - 1}^\alpha |x|^{-1-\alpha H}  dx\\
\ge C \abs{t-s}^\alpha \int_{\abs{x}<c_2/\abs{t-s}}  |x|^{\alpha (1-H)-1} dx=C\abs{t-s}^{\alpha H}.
\end{gather*}

\end{proof}

Lemmata \ref{lemma1} and \ref{lemma2} imply the following
\begin{proposition}
There exist $\delta,C_1,C_2 >0$ s.t. for rhmsp $X$ given by \eqref{X} and $|t-s|<\delta$ it holds
\begin{equation}\label{Xcont}
C_1 \abs{t-s}^{\wH(t,s)}\le \norm{X(t) - X(s)}_\alpha \le C_2 |t-s|^{\cH(t,s)},\quad t,s\in[0,T]
\end{equation}
where $\wH(t,s) = \min_{[t,s]} H(u)$, $\cH(t,s) = \max_{[t,s]} H(u)$.
\end{proposition}
\begin{proof}
Let $\cH(t,s) = H(\check t)$, $\wH(t,s)=H(\hat t)$.
\begin{gather*}
\norm{X(t) - X(s)}_\alpha\\
\le \norm{Z^{H(t)}(t) - Z^{\cH(t,s)}(t)}_\alpha +\norm{Z^{H(s)}(s) - Z^{\cH(t,s)}(s)}_\alpha + \norm{Z^{\cH(t,s)}(t) - Z^{\cH(t,s)}(s)}_\alpha\\
\le \abs{H(t)-H(\check t)}+ \abs{H(s)-H(\check t)} + C|t-s|^{\cH(t,s)}\le C|t-s|^\gamma + C|t-s|^{\cH(t,s)}
\end{gather*}
Since $\cH(t,s)<\gamma$, we get the upper bound.

The lower one is proved similarly:
\begin{gather*}
C\abs{t-s}^{\wH(t,s)}\le \norm{Z^{\wH(t,s)}(t) - Z^{\wH(t,s)}(s)}_\alpha\\\le \norm{X(t) - X(s)}_\alpha +\norm{Z^{H(t)}(t) - Z^{\wH(t,s)}(t)}_\alpha +\norm{Z^{H(s)}(s) - Z^{\wH(t,s)}(s)}_\alpha \\
\le \norm{X(t) - X(s)}_\alpha + \abs{H(t)-H(\hat t)}+ \abs{H(s)-H(\hat t)}\\\le \norm{X(t) - X(s)}_\alpha+C|t-s|^\gamma.
\end{gather*}
\end{proof}
\subsection{H\"older continuity of rhmsp}
In this subsection we prove a H\"older continuity of rhmsp. Our argument is a slight modification of the one found in \cite{kono-maejima}
for harmonizable fractional stable motion. A similar argument was also used in \cite{lacaux09} to prove a H\"older regularity
of operator scaling stable random fields.
\begin{theorem}
The rhmsp $X$ has a version, which is almost surely H\"older continuous of any order $\kappa<\wH$ and
moreover almost surely satisfies
$$
\sup_{\substack{t,s\in[0,T]\\|t-s|<\delta}} \abs{X(t)-X(s)} = o(\delta^{\wH} \abs{\log\delta}^{1/\alpha+1/2+\varepsilon}),\quad \delta\to 0+,
$$
for all $T,\varepsilon>0$.
\end{theorem}
\begin{proof}
Let $T>0$ be fixed and throughout this proof $t,s\in[0,T]$.

We use the LePage representation \eqref{lepage}. To simplify the notation we
write this representation for the process $X$ itself rather than for its version:
$$
X(t) = C_\alpha \re \sum_{k\ge 1} \Gamma_k^{-1/\alpha}\varphi(\xi_k)^{-1/\alpha}f(t,\xi_k) g_k,
$$
where $f(t,x) = (e^{itx}-1)\abs{x}^{-1/\alpha-H(t)}$, $\varphi(x) = K_\eta\abs{x}^{-1}\abs{\log\abs{x}}^{-1-\eta}$,
$\eta>0$ is arbitrary but fixed, $K_\eta$ is a normalizing constant.

Conditioning on $\Gamma$ and $\xi$, $X$ has the Gaussian distribution, so
\begin{gather*}
\ex{(X(t)-X(s))^2\mid \Gamma,\xi} = C_\alpha^2\sum_{k\ge 1} \Gamma_k^{-2/\alpha} \varphi(\xi_k)^{-2/\alpha}\abs{f(t,\xi_k)-f(s,\xi_k)}^2
\le C a(u),
\end{gather*}
where
\begin{gather*}
a(u) = \sum_{k\ge 1} \Gamma_k^{-2/\alpha} \varphi(\xi_k)^{-2/\alpha}\sup_{\abs{t-s}<u}\abs{f(t,\xi_k)-f(s,\xi_k)}^2.
\end{gather*}
Write
\begin{gather*}
\sup_{\abs{t-s}<u}\abs{f(t,\xi_k)-f(s,\xi_k)}\\
\le \sup_{\abs{t-s}<u}\abs{(e^{itx}-e^{isx}}\abs{x}^{-1/\alpha-H(t)} +
\sup_{\abs{t-s}<u}\abs{e^{isx}-1}\abs{x}^{-1/\alpha}\abs{x^{-H(t)}-x^{-H(s)}}\\
\le C (u\abs x \wedge 1)\abs{x}^{-1/\alpha}(\abs{x}^{-\wH}\vee \abs{x}^{-\cH})
\\+ C (\abs x \wedge 1)\abs{x}^{-1/\alpha}(\abs{x}^{-\wH}\vee \abs{x}^{-\cH})\abs{\log \abs{x}} \sup_{\abs{t-s}<u}\abs{H(t)-H(s)}\\
\le C \abs{x}^{-1/\alpha}(\abs{x}^{-\wH}\vee \abs{x}^{-\cH})\big((u\abs x \wedge 1)+ (\abs x \wedge 1)\abs{\log \abs{x}} u^\gamma \big).
\end{gather*}
Keeping this estimate in mind, take now the expectation $\exxi{a(u)}$ with respect to the variables $\xi$ only:
\begin{gather*}
\exxi{a(u)}\le C S(\Gamma) (I_1+I_2),
\end{gather*}
where
\begin{gather*}
I_1=\int_\R \abs{x}^{-2/\alpha}(\abs{x}^{-2\wH}\vee \abs{x}^{-2\cH})
(u\abs x \wedge 1)^2 \varphi^{1-2/\alpha}(x)dx\\
= 2\int_0^\infty x^{-1}(x^{-2\wH}\vee x^{-2\cH})
(u x \wedge 1)^2 \abs{\log x}^{(1+\eta)(2/\alpha-1)}dx\\
\le C u^{2\wH} \int_0^\infty z^{-1}(z^{-2\wH}\vee z^{-2\cH})
(\abs z \wedge 1)^2  \abs{\log(z/u)}^{(1+\eta)(2/\alpha-1)} dz\\
\le C u^{2\wH} \abs{\log u}^{(1+\eta)(2/\alpha-1)},
\\
I_2= u^{2\gamma}\int_\R \abs{x}^{-2/\alpha}(\abs{x}^{-2\wH}\vee \abs{x}^{-2\cH})
 (\abs x \wedge 1)^2\abs{\log \abs{x}}^2 \varphi^{1-2/\alpha}(x)dx \\
 = 2 u^{2\gamma}\int_0^\infty x^{-1}(x^{-2\wH}\vee x^{-2\cH})
(x \wedge 1)^2 \abs{\log x}^{(1+\eta)(2/\alpha-1)+2}dx \le C u^{2\gamma},
\\
S(\Gamma) = \sum_{k\ge 1} \Gamma^{-2/\alpha}_k<\infty\quad\text{a.a. $\Gamma$},
\end{gather*}
where the last is true owing to the fact that $\Gamma_j/j\to 1$, $j\to\infty$, almost surely
by the strong law of large numbers, and $2/\alpha>1$.
Therefore
\begin{gather*}
\exxi{a(u)}\le C(\Gamma) u^{2\wH} \abs{\log u}^{(1+\eta)(2/\alpha-1)}
\end{gather*}
almost surely.

Define $b(u)=u^{2\wH} \abs{\log u}^{2(1+\eta)/\alpha}$. We have
$$
\exxi{\sum_{n\ge 1} \frac{a(2^{-n})}{b(2^{-n})}}\le C(\Gamma) \sum_{n\ge 1} n^{-1-\eta},
$$
so for almost all $\xi,\Gamma$ we have ${a(2^{-n})}/{b(2^{-n})}\to 0$, $n\to\infty$.
It is easy to see that $b(2t)\le C b(t)$, and $a(u)$ is increasing, so from the last convergence we get
$a(u)/b(u)\to 0$, $u\to 0+$, or $a(u) = o_{\xi,\Gamma}(u^{2\wH} \abs{\log u}^{2(1+\eta)/\alpha})$.
So we have
$$
\ex{(X(t)-X(s))^2\mid \Gamma,\xi} = o_{\xi,\Gamma}(u^{2\wH} \abs{\log u}^{2(1+\eta)/\alpha}),\quad u\to 0+.
$$
Now recall once more that $X$ is Gaussian given $\xi$ and $\Gamma$, so by Lemma 1 of \cite{kono-maejima}
$$
\sup_{\abs{t-s}<\delta}\abs{X(t)-X(s)} = o_{\omega}(\delta^{\wH} \abs{\log \delta}^{1/\alpha + \eta/\alpha + 1/2}),\quad \delta\to 0+,
$$
whence we get the statement of the theorem.
\end{proof}

\subsection{Localizability of rhmsp}
We start this section by giving Falconer's notion of localizability.
\begin{definition}
Process $X$ is called $H$-\emph{localizable} at a point $t$ with the local version $Y$ if
\begin{equation}\label{localizability}
\set{\frac{1}{\delta^H}(X(t+\delta u)-X(t)),u\ge 0}\overset{fdd}{\longrightarrow} \set{Y(u),u\ge 0},\ \delta\to 0+.
\end{equation}
(Here $\overset{fdd}{\longrightarrow}$ stands for the convergence of finite-dimensional distributions.)

It is called strongly $H$-localisable at a point $t$ if in \eqref{localizability} the convergence is in the sense of the distribution on the path space.
\end{definition}
Some authors use the term \emph{local asymptotic self-similarity} for localizability, which reflects the fact that the local version $Y$ is an $H$-self-similar process.
\begin{theorem}
The rhmsp $X$  is localizable at any point $t$ with local version being real harmonizable fractional stable process with Hurst parameter $H(t)$.
\end{theorem}
\begin{proof}
Define
$$
Y_t^\delta(u) = \frac{1}{\delta^{H(t)}}(X(t+\delta u)-X(t)).
$$
We will assume throughout that $\delta<1$.

For $u_1,\dots,u_n>0$, $\la_1,\dots,\la_n\in \R$ denote $s_k=t+\delta u_k$ and write
\begin{gather*}
-\log \ex{\exp\set{i\sum_{k=1}^n\lambda_k Y_t^\delta(u_k)}}\\
= -\log \ex{\exp\set{\frac{i}{\delta^{H(t)}}\int_\R\frac{1}{\abs{x}^{1/\alpha}}
\sum_{k=1}^n\lambda_k \Big[\frac{e^{i s_k x}-1}{\abs{x}^{H(s_k)}}-\frac{e^{itx}-1}{\abs{x}^{H(t)}}\Big]M(dx)}}\\
= \set{\frac{1}{\delta^{\alpha H(t)}}\int_\R\frac{1}{\abs{x}} \abs{\sum_{k=1}^n\lambda_k \Big[\frac{e^{i s_k x}-1}{\abs{x}^{H(s_k)}}-\frac{e^{itx}-1}{\abs{x}^{H(t)}}\Big]}^\alpha dx}.
\end{gather*}
Now estimate the integrand multiplied by $|x|$:
\begin{gather*}
\abs{\sum_{k=1}^n\lambda_k \Big[\frac{e^{i s_k x}-1}{\abs{x}^{H(s_k)}}-\frac{e^{itx}-1}{\abs{x}^{H(t)}}\Big]}^\alpha
\le C \sum_{k=1}^n \abs{\la_k}^\alpha\abs{\frac{e^{i s_k x}-1}{\abs{x}^{H(s_k)}}-\frac{e^{itx}-1}{\abs{x}^{H(t)}}}^\alpha\\
\le C \sum_{k=1}^n \abs{\la_k}^\alpha \left(\abs{e^{i s_k x}-1}^\alpha\abs{\abs{x}^{-H(s_k)}-\abs{x}^{-H(t)}}^\alpha+\abs{x}^{-\alpha H(t)}\abs{e^{i s_k x}-e^{itx}}^\alpha\right)\\
\le C \sum_{k=1}^n \abs{\la_k}^\alpha \left(\abs{\log\abs{x}}^\alpha \abs{x}^{-\alpha \theta}\abs{H(s_k)-H(t)}^{\alpha}(1\wedge\abs{x})^\alpha+\abs{x}^{-\alpha H(t)}\abs{s_k -t}^\alpha (1\wedge\abs{x})^\alpha\right)\\
\le C \sum_{k=1}^n \abs{\la_k}^\alpha (\abs{x}^{-\alpha \cH}\vee \abs{x}^{-\alpha \wH})(1\wedge\abs{x})^\alpha\big[\abs{s_k-t}^{\alpha \gamma} \abs{\log{\abs{x}}}^\alpha +\abs{s_k-t}^\alpha\big]\\
\le C \sum_{k=1}^n \abs{\la_k}^\alpha \delta^{\alpha \gamma} \big(\abs{x}^{\alpha(1-\wH)}\ind{\abs{x}<1}+ \abs{x}^{-\alpha \cH}\ind{\abs{x}>1}\big)(1+\abs{\log{\abs{x}}}^\alpha)\\
\le C\delta^{\alpha \gamma} \big(\abs{x}^{\alpha(1-\wH)}\ind{\abs{x}<1}+ \abs{x}^{-\alpha \cH}\ind{\abs{x}>1}\big)(1+\abs{\log{\abs{x}}}^\alpha).
\end{gather*}
We remark that the constants here depend only on $\alpha$, $t,u_1,\dots,u_n$ and $\la_1,\dots,\la_n$.

Now
\begin{gather*}
\frac{1}{\delta^{\alpha H(t)}}\frac{1}{\abs{x}} \abs{\sum_{k=1}^n\lambda_k \Big[\frac{e^{i s_k x}-1}{\abs{x}^{H(s_k)}}-\frac{e^{itx}-1}{\abs{x}^{H(t)}}\Big]}^\alpha\\
\le C\big(\abs{x}^{\alpha(1-\wH)-1}\ind{\abs{x}<1}+ \abs{x}^{-1-\alpha \cH}\ind{\abs{x}>1}\big)(1+\abs{\log{\abs{x}}}^\alpha),
\end{gather*}
which is integrable over $\R$. Hence by the dominated convergence theorem
\begin{gather*}
-\lim_{\delta\to 0+}\log \ex{\exp\set{i\sum_{k=1}^n\lambda_k Y_t^\delta(u_k)}}\\
= \int_\R\frac{1}{\abs{x}} \lim_{\delta\to 0+}\set{\frac{1}{\delta^{\alpha H(t)}}\abs{\sum_{k=1}^n\lambda_k \Big[\frac{e^{i s_k x}-1}{\abs{x}^{H(s_k)}}-\frac{e^{itx}-1}{\abs{x}^{H(t)}}\Big]}^\alpha dx}\\
= \int_\R \lim_{\delta\to 0+}\frac{1}{\abs{\delta x}} \abs{\sum_{k=1}^n\lambda_k \frac{1}{\delta^{H(t)}}\Big[\frac{e^{i (t+\delta u_k) x}-1}{\abs{x}^{H(t+\delta u_k)}}-\frac{e^{itx}-1}{\abs{x}^{H(t)}}\Big]}^\alpha d(\delta x)\\
= \int_\R\frac{1}{\abs{y}} \abs{\sum_{k=1}^n\lambda_k \lim_{\delta\to 0+}\Big[\frac{e^{i u_k y}-e^{-ity/\delta}}{\abs{y}^{H(t+\delta u_k)}}\delta^{H(t+\delta u_k)-H(t)}-\frac{1-e^{-ity/\delta}}{\abs{y}^{H(t)}}\Big]}^\alpha dy\\
= \int_\R\frac{1}{\abs{y}} \abs{\sum_{k=1}^n\lambda_k \lim_{\delta\to 0+}\Big[\frac{e^{i u_k y}-1}{\abs{y}^{H(t)}}+(e^{itu_k} -e^{-ity/\delta})R_\delta\Big]}^\alpha dy,
\end{gather*}
where
\begin{gather*}
R_\delta = \frac{1}{\abs{y}^{H(t+\delta u_k)}}\delta^{H(t+\delta u_k)-H(t)}-\frac{1}{\abs{y}^{H(t)}} = \frac{1}{|y|^{H(t)}}\Big[\Big(\frac{\delta}{\abs{y}}\Big)^{H(t+\delta u_k)-H(t)}-1\Big].
\end{gather*}
Estimate
\begin{gather*}
\abs{\log\Big(\frac{\delta}{\abs{y}}\Big)^{H(t+\delta u_k)-H(t)}}\le \abs{H(t+\delta u_k)-H(t)}(\abs{\log \delta}+\abs{\log\abs{y}})\\
\le C\delta^\gamma (\abs{\log \delta}+\abs{\log\abs{y}})\to 0, \quad \delta\to 0+.
\end{gather*}
Thus $R_\delta \to 0$, $\delta\to 0+$.

Finally,
\begin{equation}\label{cfconv}
\lim_{\delta\to 0+}\log \ex{\exp\set{i\sum_{k=1}^n\lambda_k Y_t^\delta(u_k)}}
= -\int_\R\frac{1}{\abs{y}} \abs{\sum_{k=1}^n\lambda_k \Big[\frac{e^{i u_k y}-1}{\abs{y}^{H(t)}}\Big]}^\alpha dy,
\end{equation}
which is exactly the logarithm of the characteristic function of $Z^{H(t)}(u_1),\dots$,$Z^{H(t)}(u_n)$, as required.
\end{proof}
\begin{remark}
By using the same kind of argument as the one used in the proof of continuity, it is possible to prove tightness of laws
of processes on the space of continuous paths and whence derive a strong localizability.
\end{remark}

\section{Local times for rhmsp}

\subsection{Properties of the local time}

We start this section by showing the existence and square integrability of a local time.
\begin{proposition}
The rhmsp $X$ has a square integrable local time $L(t,x)$.
\end{proposition}
\begin{proof}
According to \cite{dozzietal}, it is enough to check the following ``condition ($\mathcal{H}$)'': there exists $\rho>0$ and $H\in(0,1)$ and $\psi\in L^1(\mathbb R)$ such that for all $\abs{t-s}<\rho$
\begin{equation}
\label{condH}
\abs{\ex{\exp\set{i\lambda\big({X(t)-X(s)}\big)}}}\le \psi(\la\abs{t-s}^H).
\end{equation}

But, in view of \eqref{Xcont}, for $t$ and $s$ close enough
\begin{gather*}
\abs{\ex{\exp\set{i\lambda\big({X(t)-X(s)}\big)}}}= \exp\set{-\abs{\la}^\alpha \norm{X(t)-X(s)}_\alpha^\alpha}
\\
\le \exp\set{-C\abs{\lambda}^\alpha \norm{t-s}^{\alpha\cH(t,s)}}\le \exp\set{-C\abs{\lambda}^\alpha \norm{t-s}^{\alpha\wH}},
\end{gather*}
whence we have \eqref{condH} with $\psi = \exp\set{-C{\abs{x}^\alpha}}$, $H=\wH$.
\end{proof}
In order to prove further properties, we need
\begin{definition}[\cite{nolan}]
A stable random process $X$ is \emph{$\norm{\cdot}_\alpha$ locally non-deterministic} (LND) on $\mathbf{T}$ if
\begin{enumerate}[(L1)]
\item $\norm{X(t)}_\alpha >0$ for all $t\in \mathbf{T}$;
\item $\norm{X(t)-X(s)}_\alpha>0$ for all sufficiently close distinct $s,t\in\mathbf{T}$;
\item for any $n>1$ there exists $C_n>0$ s.t. for any $t_1<t_2<\dots <t_n\in \mathbf T$ sufficiently close together
one has
\begin{equation}\label{lnd}
\norm{X(t_n) - \spann\set{X(t_1),\dots,X(t_{n-1})}}_\alpha\ge C_n\norm{X(t_n)-X(t_{n-1})}_\alpha.
\end{equation}
\end{enumerate}
\end{definition}

In \cite{nolan} it is shown that the local non-determinism property is equivalent to the property of \emph{$\norm{\cdot}_\alpha$ locally approximately independent increments}, which consists of properties (L1), (L2) above and
\begin{enumerate}[(L{3}a)]
\item for any $n>1$ there exists $C_n$ s.t. for any $t_1<t_2<\dots <t_n\in \mathbf T$ sufficiently close together
and any $a_1,\dots,a_n\in\mathbb{R}$ one has
\begin{equation}\label{aii}\begin{gathered}
\norm{a_1 X(t_1) +\sum_{k=1}^{n-1} a_k \big(X(t_{k+1})-X(t_{k})\big)}_\alpha\\\ge C_n\left(\norm{a_1 X(t_1)}_\alpha +\sum_{k=1}^{n-1} \norm{a_k \big(X(t_{k+1})-X(t_{k})\big)}_\alpha\right).
\end{gathered}
\end{equation}
\end{enumerate}

\begin{theorem}\label{mainthm}
For any $\varepsilon>0$ the rhmsp $X$ is LND on $[\varepsilon,T]$.
\end{theorem}
\begin{proof}
The main difficulty is to prove property (L3) of LND, as property (L1) is obvious and property (L2) follows from \eqref{Xcont}.

We proceed in two steps.

\emph{Step 1}. We prove LND for a modification of rhmsp $X$  defined by
\begin{equation}\label{Yt}
Y(t) =  \int_{\R} (1-e^{-i\spr{tx}})(-ix)^{-H(t)-1/\alpha} M(dx),
\end{equation}
where
$$
(-ix)^{-K} = \abs{x}^{-K} e^{i\pi K \sign x/2 }.
$$
The Fourier transform of the function $f_Y(t,x) = (1-e^{-i\spr{tx}})(-ix)^{-H(t)-1/\alpha}$ (w.r.t.\ the second variable)
on $L^\alpha(\R)$ is
\begin{equation}
\widehat{f_Y}(t,x) = \frac{1}{\Gamma(H(t)+1/\alpha)}\left((t-x)_+^{H(t)-1/\beta}-(-x)_+^{H(t)-1/\beta}\right),
\end{equation}
where $\beta = \alpha/(\alpha-1)$ is the exponent  adjoint to $\alpha$, see Lemma~\ref{ftlema}.

In order to check property (L3) for $Y$, we have to find a good lower bound to
\begin{gather*}
\norm{Y(t_n) - \sum_{k=1}^{n-1} u_k Y(t_k)}_\alpha =
\norm{f_Y(t_n,\cdot) - \sum_{k=1}^{n-1} u_k f_Y(t_k,\cdot)}_{L^\alpha(\R)}.
\end{gather*}
It is fortunately given by the Hausdorff-Young inequality:
\begin{gather*}
\norm{f_Y(t_n,\cdot) - \sum_{k=1}^{n-1} u_k f_Y(t_k,\cdot)}_{L^\alpha(\R)}\ge
C \norm{ \widehat{f_Y}(t_n,\cdot) - \sum_{k=1}^{n-1} u_k \widehat{f_Y}(t_k,\cdot)}_{L^\beta(\R)}\\
\ge C \norm{ \widehat{f_Y}(t_n,\cdot)}_{L^\beta([t_{n-1},t_n])}= C\left(\frac{1}{\Gamma(H(t_n)+1/\alpha)^\beta}\int_{t_{n-1}}^{t_n} (t_n-x)^{\beta H(t_n) - 1}dx\right)^{1/\beta}\\
\ge C (t_n-t_{n-1})^{H(t_n)}\ge C\norm{X(t_n)-X(t_{n-1})}_\alpha
\end{gather*}
for $t_n$ and $t_{n-1}$ close enough. (We have used the fact that $\widehat{f_Y}(t_k,x)$ vanishes on $[t_{n-1},t_n]$ for $k<n$ in the middle, and inequality \eqref{Xcont} in the last step.) But it is straightforward to check (see a much stronger statement below in the Step 2) that
\begin{gather*}
\norm{X(t_n)-X(t_{n-1})}_\alpha\ge
\norm{Y(t_n)-Y(t_{n-1})}_\alpha - C|H(t_n)-H(t_{n-1})|\\\ge \norm{Y(t_n)-Y(t_{n-1})}_\alpha - C|t_n-t_{n-1}|^\gamma
\ge C\norm{Y(t_n)-Y(t_{n-1})}_\alpha,
\end{gather*}
which gives the desired LND property.

\emph{Step 2}. Here we show how the property of locally asymptotically independent increments for $Y$ implies that for $X$.
Denote $f_X(t,x) = (e^{itx}-1)\abs{x}^{-H(t)-1/\alpha}$, $f_{1}(t,x) = - f_Y(t,-x) = f_X(t,x) e^{-i\pi (H(t)+1/\alpha) \sign x/2 }$ and
write for $0 < t_1<t_2<\dots <t_n<T$ and $a_1,\dots,a_n\in \R$ (we put $t_0=0$ for the sake of simplicity)
\begin{gather*}
\norm{\sum_{k=1}^{n-1} a_k \big(X(t_{k+1})-X(t_{k})\big)}_\alpha
= \norm{\sum_{k=0}^{n-1} a_k \big(f_X(t_{k+1},\cdot)-f_X(t_k,\cdot)\big)}_{L^\alpha(\R)}\\
= \norm{ e^{-i \pi (H(t_1)+1/\alpha)\sign x/2}\sum_{k=0}^{n-1} a_k \big(f_X(t_{k+1},\cdot)-f_X(t_k,\cdot)\big)}_{L^\alpha(\R)}\\
\ge \norm{ \sum_{k=0}^{n-1} a_k \big(f_1(t_{k+1},\cdot)-f_1(t_k,\cdot)\big)}_{L^\alpha(\R)}
- \sum_{k=0}^{n-1}\norm{a_k \big(\Delta(t_{k+1},\cdot)-\Delta(t_k,\cdot)\big)}_{L^\alpha(\R)}\\
\end{gather*}
where \begin{gather*}
\Delta(t,x) = e^{-i \pi (H(t_1)+1/\alpha)\sign x/2} f_X(t,x) - f_1(t,x)\\ = e^{-i\pi\sign x/(2\alpha)}(e^{-i\pi H(t_1) \sign x/2}-e^{-i\pi H(t) \sign x/2}) f_X(t,x).
\end{gather*}
Estimate
\begin{gather*}
\norm{\Delta(t_{k+1},\cdot) - \Delta(t_{k},\cdot)}_{L^\alpha(\R)}\\\le \norm{(e^{-i\pi H(t_{k+1}) \sign x/2}-e^{-i\pi H(t_k) \sign x/2}) f_X(t_{k+1},\cdot)}_{L^\alpha(\R)}
\\+ \norm{(e^{-i\pi H(t_k) \sign x/2}-e^{-i\pi H(t_1) \sign x/2}) \big(f_X(t_{k+1},\cdot)-f_X(t_k,\cdot)\big)}_{L^\alpha(\R)}\\
\le C\abs{H(t_{k+1})-H(t_k)} \norm{f_X(t_{k+1},\cdot)}_{L^\alpha(\R)} \\+ C\abs{H(t_{k+1})-H(t_1)}\norm{f_X(t_{k+1},\cdot)-f_X(t_k,\cdot)\big)}_{L^\alpha(\R)}\\
\le C\abs{t_{k+1}-t_k}^{\gamma} \norm{X(t_{k+1})}_\alpha + C \abs{t_{k+1}-t_1}^{\gamma} \norm{X(t_k)-X(t_k+1)}_\alpha
\\= o\big(\norm{X(t_{k+1})-X(t_k)}_\alpha\big), \quad \abs{t_n-t_1}\to \infty.
\end{gather*}
Further,
\begin{gather*}
\norm{ \sum_{k=0}^{n-1} a_k \big(f_1(t_{k+1},\cdot)-f_1(t_k,\cdot)\big)}_{L^\alpha(\R)}
=\norm{ \sum_{k=0}^{n-1} a_k \big(f_Y(t_{k+1},\cdot)-f_Y(t_k,\cdot)\big)}_{L^\alpha(\R)} \\
=\norm{ \sum_{k=0}^{n-1} a_k \big(Y(t_{k+1})-Y(t_k)\big)}_{\alpha}\ge  C\sum_{k=0}^{n-1} \norm{a_k \big(Y(t_{k+1})-Y(t_k)\big)}_{\alpha},
\end{gather*}
where the last inequality is true thanks to LND property of $Y$. Similarly to the first estimate of Step 2,
\begin{gather*}
\norm{Y(t_{k+1})-Y(t_k)}_{\alpha} = \norm{f_Y(t_{k+1},\cdot)-f_Y(t_k,\cdot)}_{L^\alpha(\R)} = \norm{f_1(t_{k+1},\cdot)-f_1(t_k,\cdot)}_{L^\alpha(\R)}\\\ge \norm{X(t_{k+1})-X(t_k)}_{\alpha} - \norm{\Delta(t_{k+1},\cdot) - \Delta(t_{k},\cdot)}_{L^\alpha(\R)}\\
\ge \norm{X(t_{k+1})-X(t_k)}_{\alpha} - o\big(\norm{X(t_{k+1})-X(t_k)}_\alpha\big),
\end{gather*}
so finally
\begin{gather*}
\norm{\sum_{k=1}^{n-1} a_k \big(X(t_{k+1})-X(t_{k})\big)}_\alpha\\\ge \sum_{k=1}^{n-1} \abs{a_k}\Big(\norm{X(t_{k+1})-X(t_{k})}_\alpha- o\big(\norm{X(t_{k+1})-X(t_{k})}_\alpha\big)\Big)\\
\ge C\sum_{k=1}^{n-1} \abs{a_k}\norm{\big(X(t_{k+1})-X(t_{k})\big)}_\alpha
\end{gather*}
for $\abs{t_n-t_1}$ small enough.
\end{proof}
\begin{remark}
It is possible to make the presented proof shorter by skipping several lines at the end of Step 1: in fact the interim lower estimate by $\norm{X(t_n)-X(t_{n-1})}_\alpha$ is exactly what is needed in proof, and one does not need to go further obtaining the LND for $Y$. Nevertheless,
these lines makes the proof more structured, and we think that this intermediate result of LND for $Y$ is interesting on its own.
\end{remark}

Thanks to \cite[Theorem 4.1]{nolan} and estimates for the norms of increments of rhmsp $X$ we have whence the following result.
\begin{theorem}\label{mainthm-1}
The local time $L(t,x)$ of the rhmsp $X$ is jointly continuous in $(t,x)$ for $t>0$, moreover, for any $\kappa<(1/\cH-1)/2$ it is $\kappa$-H\"older continuous in $x$.
\end{theorem}

\begin{remark}[Gaussian case]
It is not hard to see that all results of this paper extend to the Gaussian case, viz $\alpha=2$. In fact all the proofs work in the
Gaussian case as well, except the one of the pathwise continuity, but there one has a much shorter and direct proof. Also, in the
proof of Theorem~\ref{mainthm} one does not need the Hausdorff--Young inequality, but just Parseval's identity.

It should be noted that some papers mistakenly claim that in Gaussian case, the real harmonizable multifractional Brownian motion
is up to a multiplicative constant equivalent to moving average (or linear) multifractional Brownian motion, as defined e.g.\ in \cite{dozzietal}.
The difference between the two definitions is shown in \cite{dobricojeda} to be essential. Nevertheless, this claim is true for the process $Y$
defined in the proof of  Theorem~\ref{mainthm}.
\end{remark}

\appendix
\section{Fourier transform} \renewcommand{\thesection}{A}
In this appendix we compute the Fourier transform which is used by many authors, however, we were not able
to find a rigorous derivation. Below we define the Fourier transform
$$
\widehat f(u) = \int_{\R} e^{i u x} f(x) dx,
$$
and use the notation $x_+ = x\vee 0$.
\begin{lemma}\label{ftlem}
For $h\in(1,2) $, $t>0$ the Fourier transform of
$$
f_{h,t}(x) = (1-e^{-itx})(-ix)^{-h} = (1-e^{-itx}) \abs{x}^{-h} e^{i\pi h \sign x/2 }
$$
is
$$
\widehat{f_{h,t}}(u) = \frac{2\pi}{\Gamma(h)}\big((t-u)_+^{h-1}-(-u)_+^{h-1}).
$$
\end{lemma}
\begin{proof}
First note that $z^{-h}=\frac{1}{\Gamma(h)}\int_0^\infty e^{-v z} v^{h-1} dv$ is analytic
for $\re z >0$. So it follows from the operational calculus that for any $a>0$
\begin{gather*}
\frac{v_+^{h-1}}{\Gamma(h)} = \frac{1}{2\pi i}\int_{a-i\infty}^{a+i\infty} e^{vz}z^{-h} dz= -\frac{1}{2\pi} \int_{-\infty -ai}^{\infty - ai}  e^{-ivy}(-iy)^{-h} dy,
\end{gather*}
where we have changed the variable $z\to iy$.
Plugging $v=-u$ and $v=t-u$ to this identity, we get
\begin{gather}\label{ft}
\frac{1}{\Gamma(h)}\big((t-u)_+^{h-1}-(-u)_+^{h-1}) = \frac{1}{2\pi} \int_{-\infty -ai}^{\infty - ai}  e^{iuy}(1-e^{-ity})(-iy)^{-h}  dy,
\end{gather}
Now let in this integral $y= x-ai$, $x\in \R$ and estimate for $a\in(0,1)$ the integrand as
\begin{equation}
\label{ftineq}
\begin{gathered}
\abs{e^{iuy}(1-e^{-ity})(-iy)^{-h}} = e^{au}\abs{1-e^{-at -itx}} \abs{y}^{-h}\le C(u) (t\abs{x-ai}\wedge 1 )\abs{y}^{-h} \\= C(u) (t\abs{y}^{-h+1}\wedge |y|^{-h})\le C(u) (t\abs{x}^{-h+1}\wedge |x|^{-h}),
\end{gathered}
\end{equation}
which is integrable due to the assumption $h\in(1,2)$. So letting $a\to 0+$ in \eqref{ft} yields the desired result by the dominated
convergence theorem.
\end{proof}
By the Hausdorff-Young inequality (see \cite[Theorem 5.7]{Lieb}), for $\alpha\in[1,2]$ the Fourier transform from $L^1(\R)\cap L^\alpha(\R)$ can be extended to a
bounded linear operator $\F_\alpha\colon L^\alpha(\R)\to L^\beta(\R)$, where $\beta=\alpha/(\alpha-1)$ is the exponent adjoint
to $\alpha$. We will call this map a \emph{Fourier transform on $L^\alpha(\R)$}, and we emphasize once more its boundedness
due to the Hausdorff-Young inequality:
\begin{equation}\label{hyi}
\norm{\F_\alpha f}_{L^\beta(\R)}\le C_\alpha \norm{f}_{L^\alpha(\R)}.
\end{equation}
The following lemma is an $L^\alpha(\R)$ analogue of Lemma \ref{ftlem}.
\begin{lemma}\label{ftlema}
For $\alpha\in(1,2)$, $h\in (1/\alpha,1+1/\alpha)$ and $t>0$ the Fourier transform on $L^\alpha$ of
$$
f_{h,t}(x) = (1-e^{-itx})(-ix)^{-h} = (1-e^{-itx}) \abs{x}^{-h} e^{i\pi h \sign x/2 }
$$
is
$$
\F_\alpha f_{h,t}(u) = \frac{2\pi}{\Gamma(h)}\big((t-u)_+^{h-1}-(-u)_+^{h-1}).
$$
\end{lemma}
\begin{proof}
Repeat the proof of the previous lemma to inequality \eqref{ftineq} and raise it to the power $\alpha$:
$$
\abs{e^{iuy}(1-e^{-ity})(-iy)^{-h}}^\alpha\le C(u)^\alpha (t\abs{x}^{\alpha(1-h)}\wedge |x|^{-\alpha h}),
$$
which is integrable for $h\in (1/\alpha,1+1/\alpha)$. So the integrands in \eqref{ft} converge as $a\to 0+$
in $L^\alpha(\R)$ to $f_{h,t}(x)$ by the dominated convergence theorem, hence by continuity of $\F_\alpha$ on $L^\alpha(\R)$ we get the
statement of the lemma.
\end{proof}

\bibliographystyle{elsarticle-harv}
\bibliography{dozzi-shev.version02.12.10}
\end{document}